\documentclass[11pt]{amsart}

\theoremstyle{plain}
\newtheorem{theorem}{Theorem}[section]
\newtheorem{lemma}[theorem]{Lemma}

\theoremstyle{definition}
\newtheorem{definition}[theorem]{Definition}
\newtheorem{example}[theorem]{Example}

\begin{document}

\title[Existence of Three Positive Solutions]{Existence of
Three Positive Solutions to Some $p$-Laplacian Boundary Value Problems}

\thanks{This is a preprint of a paper whose final and definite form
will be published in \emph{Discrete Dynamics in Nature and Society}.
Submitted 24-March-2012; revised 05-Aug-2012; accepted 19-Oct-2012.}


\author[M. R. Sidi Ammi]{Moulay Rchid Sidi Ammi}

\address{AMNEA Group, Department of Mathematics,
Faculty of Sciences and Technics,
Moulay Ismail University,
B.P. 509, Errachidia, Morocco}

\email{sidiammi@ua.pt}


\author[D. F. M. Torres]{Delfim F. M. Torres}

\address{Center for Research and Development in Mathematics and Applications,
Department of Mathematics, University of Aveiro, 3810-193 Aveiro, Portugal}

\email{delfim@ua.pt}


\begin{abstract}
We obtain, by using the Leggett--Williams fixed point theorem,
sufficient conditions that ensure the existence of at least
three positive solutions to some $p$-Laplacian
boundary value problems on time scales.
\end{abstract}




\maketitle


\section{Introduction}

The study of dynamic equations on time scales goes back to
the 1989 Ph.D. thesis of Stefan Hilger \cite{h1,h2},
and is currently an area of mathematics receiving
considerable recent attention \cite{abra,china-Xiamen,MyID:224,MR2794990,DDNS:12}.
Although the basic aim of the theory of time scales
is to unify the study of differential
and difference equations in one and the same subject,
it also extends these classical domains
to hybrid and ``in between'' cases. A great deal of work has
been done since the eighties of the XX century in unifying
the theories of differential and difference equations
by establishing more general results in the time scale setting
\cite{a3,b1,b2,ID:240735,r:1}.

Boundary value $p$-Laplacian problems
for differential equations and finite difference equations
have been studied extensively (see, \textrm{e.g.}, \cite{agra} and references therein).
Although many existence results for dynamic equations on time scales
are available \cite{and,kau}, there are not many results concerning
$p$-Laplacian problems on time scales \cite{anderson,jipam,ijpam,da}.
In this paper we prove new existence results for three classes of
$p$-Laplacian boundary value problems on time scales. In contrast with
our previous works \cite{jipam} and \cite{ijpam},
which make use of the Krasnoselskii fixed point theorem
and the fixed point index theory, respectively,
here we use the Leggett--Williams fixed point theorem
\cite{k1,zhao} obtaining multiplicity of positive solutions.
The application of the Leggett--Williams fixed point theorem
for proving multiplicity of solutions for boundary value
problems on time scales was first introduced
by Agarwal and O'Regan \cite{r2}, and is now recognized
as an important tool to prove existence of positive solutions
for boundary value problems on time scales \cite{r4,r5,r6,r3,r7,r8}.

The paper is organized as follows. In Section~\ref{sec:prelim}
we present some necessary results from the theory of time scales (\S\ref{subsec:prelim:ts})
and the theory of cones in Banach spaces (\S\ref{subsec:prelim:cones}).
We end \S\ref{subsec:prelim:cones} with the Leggett--Williams fixed point theorem
for a cone preserving operator, which is our main tool in proving
existence of positive solutions to the boundary value problems
on time scales we consider in Section~\ref{sec:mr}.
The contribution of the paper is Section~\ref{sec:mr}, which is divided
in three parts. The purpose of the first part (\S\ref{sec:mr:part1}) is to prove
existence of positive solutions
to the nonlocal $p$-Laplacian dynamic equation on time scales
\begin{equation}
\label{eq1}
-\left ( \phi_{p}(u^{\Delta}(t))\right)^{\nabla}=
\frac{\lambda f(u(t))}{( \int_{0}^{T} f(u(\tau ))\, \nabla \tau )^{2}},
\quad  t \in  (0,T)_{\mathbb{T}} \, ,
\end{equation}
satisfying the boundary conditions
\begin{equation}
\label{eq2}
\begin{gathered}
\phi_{p}(u^{\Delta}(0)) - \beta \left (
\phi_{p}(u^{\Delta}(\eta))\right)=0,\\
u(T) -\beta u(\eta)=0,
\end{gathered}
\end{equation}
where $\eta \in (0,T)_{\mathbb{T}}$,
$\phi_{p}(\cdot)$ is the $p$-Laplacian operator defined by
$\phi_{p}(s)= |s|^{p-2} s$, $p>1$, and $(\phi_{p})^{-1}= \phi_{q}$ with $q$
the Holder conjugate of $p$, \textrm{i.e.}, $\frac{1}{p}+ \frac{1}{q}= 1$.
The concrete value of $p$ is connected with the application at hands.
For $p=2$, for example, problem \eqref{eq1}--\eqref{eq2} describes the operation
of a device flowed by an electric current, \textrm{e.g.}, thermistors \cite{sd},
which are devices made from materials whose electrical conductivity
is highly dependent on the temperature. Thermistors have the advantage
of being temperature measurement devices of low cost, high resolution,
and flexible in size and shape. Constant $\lambda$ in \eqref{eq1} is a dimensionless parameter
that can be identified with the square of the applied potential difference at the ends
of a conductor; $f(u)$ is the temperature dependent resistivity of
the conductor; and $\beta$ in \eqref{eq2} is a transfer coefficient supposed to verify
$0< \beta < 1$. For a more detailed discussion about the physical justification of
equations \eqref{eq1}--\eqref{eq2} the reader is referred to \cite{jipam}.
Theoretical analysis (existence, uniqueness, regularity, and asymptotic results)
for thermistor problems with various types of boundary
and initial conditions have received significant attention in the last few years
for the particular case $\mathbb{T} = \mathbb{R}$ \cite{ant,cim,fow,shi,xu}.
The second part of our results (\S\ref{sec:mr:part2}) is concerned with
the following quasilinear elliptic problem:
\begin{equation} \label{eqn1}
\begin{gathered}
-\left ( \phi_{p}(u^{\Delta}(t))\right)^{\nabla}= f(u(t)) +
h(t)\, , \quad  t \in  (0,T)_{\mathbb{T}} \, ,\\
u^{\Delta}(0)=0  \, ,\quad u(T) - u(\eta)=0 \, ,
\end{gathered}
\end{equation}
where $\eta \in (0,T)_{\mathbb{T}}$.
Results on existence of radially infinity many solutions to \eqref{eqn1} are proved
in the literature using: (i) variational methods,
where solutions are obtained as critical points of some energy
functional on a Sobolev space, with $f$ satisfying appropriate conditions \cite{ar,ck};
(ii) methods based on phase-plane analysis and the
shooting method \cite{et}; (iii) the technique of time
maps \cite{eg}. For $p=2$, $h \equiv 0$, and $\mathbb{T}=\mathbb{R}$, problem
\eqref{eqn1} becomes a well-known boundary-value problem of differential
equations. Our results generalize earlier works
to the case of a generic time scale $\mathbb{T}$, $p\neq 2$,
and $h$ not identically zero. Finally, the third part of our
contribution (\S\ref{sec:mr:part3}) is devoted to the existence
of positive solutions to the $p$-Laplacian dynamic equation
\begin{equation}
\label{equa1}
\begin{gathered}
\left(\phi_{p}(u^{\Delta}(t))\right)^\nabla + \lambda a(t) f(u(t), u(\omega(t)))=0 \, ,
\quad t \in (0,T)_{\mathbb{T}}, \\
u(t) = \psi(t), \quad t \in [-r, 0]_{\mathbb{T}},
\quad u(0)-B_{0}(u^{\Delta}(0))=0, \quad u^{\Delta}(T)=0 \, ,
\end{gathered}
\end{equation}
on a time scale $\mathbb{T}$ such that
$0, T \in \mathbb{T}_\kappa^\kappa$,
$-r \in \mathbb{T}$ with $-r \le 0 < T$,
and where $\lambda > 0$.
This problem is considered in \cite{song} where the authors
apply the Krasnoselskii fixed point theorem to obtain
one positive solution to \eqref{equa1}. Here we use the same conditions as in \cite{song},
but applying Leggett--Williams' theorem we are able to obtain more:
we prove existence of at least three positive solutions.


\section{Preliminaries}
\label{sec:prelim}

Here we just recall the basic concepts and results
needed in the sequel. For an introduction to time scales
the reader is refereed to \cite{a1,abra,a2,a3,b1,b2}
and references therein; for a good introduction to the theory of cones
in Banach spaces we refer the reader to the book \cite{g1}.


\subsection{Time Scales}
\label{subsec:prelim:ts}

A time scale $\mathbb{T}$ is an arbitrary nonempty closed subset
of the real numbers $\mathbb{R}$. The operators $\sigma$ and $\rho$ from $\mathbb{T}$
to $\mathbb{T}$ are defined in \cite{h1,h2} as
$$
\sigma(t)=\inf\{\tau\in\mathbb{T} \ | \ \tau> t\}\in\mathbb{T}, \quad
\rho(t)=\sup\{\tau\in\mathbb{T} \ | \ \tau< t\}\in\mathbb{T} \, ,
$$
and are called the forward jump operator and
the backward jump operator, respectively.
A point $t\in\mathbb{T}$ is
left-dense, left-scattered, right-dense, right-scattered if
$\rho (t)=t,\ \rho(t)<t,\ \sigma(t)=t,\ \sigma(t)>t$, respectively.
If $\mathbb{T}$ has a right scattered minimum $m$, define
$\mathbb{T}_{\kappa}=\mathbb{T}-\{m\}$; otherwise set
$\mathbb{T}_{\kappa}=\mathbb{T}$. If $\mathbb{T}$ has a left scattered
maximum $M$, define $\mathbb{T}^{\kappa}=\mathbb{T}-\{M\}$; otherwise set
$\mathbb{T}^{\kappa}=\mathbb{T}$. Following \cite{malina},
we also introduce the set
$\mathbb{T}_\kappa^\kappa = \mathbb{T}^{\kappa} \cap \mathbb{T}_{\kappa}$.

Let $f:\mathbb{T}\to \mathbb{R}$ and $t\in \mathbb{T}^{\kappa}$ (assume $t$ is
not left-scattered if $t=\sup\mathbb{T}$), then the delta derivative
of $f$ at the point $t$ is defined to be the number $f^{\Delta}(t)$
(provided it exists) with the property that for each $\epsilon>0$
there is a neighborhood $U$ of $t$ such that
$$
|f(\sigma(t))-f(s)-f^{\Delta}(t)(\sigma(t)-s) |\le | \sigma(t)-s|
\quad \mbox{for all } s\in U.
$$
Similarly, for $t\in \mathbb{T}_\kappa$ (assume $t$ is not right-scattered
if $t=\inf\mathbb{T}$),  the nabla derivative of $f$ at the point
$t$ is defined in \cite{a4} to be the number $f^{\nabla}(t)$
(provided it exists) with the property that for each $\epsilon >0$
there is a neighborhood $U$ of $t$ such that
$$
|f(\rho(t))-f(s)-f^{\nabla}(t)(\rho(t)-s) |\le | \rho(t)-s | \quad
\mbox{for all } s\in U.
$$
If $\mathbb{T}=\mathbb{R}$, then $f^\Delta(t)=f^\nabla(t)=f'(t)$. If
$\mathbb{T}=\mathbb{Z}$, then $f^\Delta(t)=f(t+1)-f(t)$ is the forward
difference operator while $f^\nabla(t)=f(t)-f(t-1)$ is the backward
difference operator.

A function $f$ is left-dense continuous (\textrm{i.e.}, $ld$-continuous), if $f$ is
continuous at each left-dense point in $\mathbb{T}$ and its
right-sided limit exists at each right-dense point in $\mathbb{T}$.
If $f$ is $ld$-continuous, then there exists
$F$ such that $F^{\nabla}(t)=f(t)$ for any $t \in \mathbb{T}_\kappa$.
We then introduce the nabla integral by
$$
\int^b_a f(t)\nabla t=F(b)-F(a).
$$
We define right-dense continuous ($rd$-continuous)
functions in a similar way. If $f$
is $rd$-continuous, then there exists $F$ such that
$F^{\Delta}(t)=f(t)$ for any $t \in \mathbb{T}^\kappa$,
and we define the delta integral by
$$
\int^b_a f(t)\Delta t=F(b)-F(a).
$$


\subsection{Cones in Banach Spaces}
\label{subsec:prelim:cones}

In this article $\mathbb{T}$ is a time scale
with $0\in\mathbb{T}_\kappa$ and $T\in\mathbb{T}^\kappa$.
We use $\mathbb{R}^{+}$ and $\mathbb{R}_0^{+}$ to denote, respectively,
the set of positive and nonnegative real numbers.
By $[0, T]_\mathbb{T}$ we denote the set $[0, T] \cap \mathbb{T}$.
Similarly, $(0, T)_\mathbb{T} = (0, T) \cap \mathbb{T}$.
Let $E= \mathbb{C}_{ld}([0, T]_\mathbb{T}, \mathbb{R})$. It follows that
$E$ is a Banach space with the norm $\|u\|= \max_{[0, T]_\mathbb{T}} |u(t)|$.

\begin{definition}
Let $E$ be a real Banach space. A nonempty, closed, convex set
$P \subset E$ is called a cone if it satisfies the following two conditions:
\begin{itemize}
\item[$(i)$] $u \in P$, $\lambda \geq 0$, implies  $\lambda u \in P$;
\item[$(ii)$] $u \in P$, $-u \in P$, implies $u=0$.
\end{itemize}
\end{definition}

Every cone $P \subset E$ induces an ordering in $E$ given by
$$
u \leq v \mbox{ if and only if } v-u \in P.
$$
\begin{definition}
Let $E$ be a real Banach space and  $P \subset E$ be a cone.
A function $\alpha : P \rightarrow \mathbb{R}_0^{+}$ is called
a nonnegative continuous concave functional if $\alpha$ is continuous and
$$
\alpha(tx + (1-t)y) \geq t \alpha(x)+ (1-t) \alpha(y)
$$
for all $x, y \in P$ and $0 \leq t \leq 1$.
\end{definition}

Let $a, b, r > 0$ be constants, $P_{r}= \{ u \in P \ | \  \|u\| < r \}$,
$P(\alpha, a, b)= \{ u \in P \ | \   a \leq \alpha(u), \|u\| < b\}$.
The following fixed point theorem provides the existence of at
least three positive solutions. The origin in $E$ is denoted by $\emptyset$.
The proof of the Leggett--Williams fixed point theorem
can be found in Guo and Lakshmikantham \cite{g1}
or Leggett and Williams \cite{r9}.

\begin{theorem}[Leggett--Williams' theorem]
\label{thm1}
Let $P$ be a cone in a real Banach space $E$. Let $G: \overline{P_{c}}\rightarrow \overline{P_{c}}$
be a completely continuous map and $\alpha$ a nonnegative continuous concave functional on $P$
such that $\alpha(u) \leq \|u\|$ $\forall u \in \overline{P_{c}}$. Suppose there exist
$a, b, d > 0$ with $0 < a < b <d \leq c$ such that
\begin{itemize}
\item[$(i)$] $ \{ u \in P(\alpha, b, d) \ | \  \alpha(u) > b \} \neq \emptyset$
and $ \alpha(Gu) > b$ for all $u \in P(\alpha, b, d)$;

\item[$(ii)$]  $\|Gu\| < a$ for all $u \in \overline{P}_{a}$;

\item[$(iii)$]   $\alpha(Gu) > b$ for all $u \in P(\alpha, b, c)$   with $\|Gu\| > d$.
\end{itemize}
Then $G$ has at least three fixed points $u_{1}$, $u_{2}$ and $u_{3}$ satisfying
$$
\|u_{1}\| < a \, , \quad  b < \alpha(u_{2}) \, , \quad
\|u_{3}\| > a \, , \quad \mbox{ and } \quad \alpha(u_{3}) < b \, .
$$
\end{theorem}


\section{Main Results}
\label{sec:mr}

We prove existence
of three positive solutions to different
$p$-Laplacian problems on time scales:
in \S\ref{sec:mr:part1} we study problem
\eqref{eq1}--\eqref{eq2}; in \S\ref{sec:mr:part2}
problem \eqref{eqn1}; and finally \eqref{equa1}
in \S\ref{sec:mr:part3}.


\subsection{Nonlocal Thermistor Problem}
\label{sec:mr:part1}

By a solution $u: \mathbb{T} \rightarrow \mathbb{R}$
of \eqref{eq1}--\eqref{eq2} we mean a delta differentiable
function such that $u^{\Delta}$ and
$\left ( |u^{\Delta}|^{p-2}u^{\Delta} \right )^{\nabla}$
are both continuous on $\mathbb{T}^{\kappa}_{\kappa}$ and
$u$ satisfies \eqref{eq1}--\eqref{eq2}.
We consider the following hypothesis:
\begin{itemize}
\item[(H1)] $f: \mathbb{R} \rightarrow \mathbb{R}^{+}$ is a
continuous function.
\end{itemize}

\begin{lemma}[Lemma~3.1 of \cite{jipam}]
\label{lm1}
Assume that hypothesis $(H1)$ on function $f$ is satisfied.
Then $u$ is a solution to \eqref{eq1}--\eqref{eq2} if and only if
$u \in E$ is a solution to the integral equation
$$
u(t)= -\int_{0}^{t}\phi_{q}\left (g(s) \right) \Delta s + B,
$$
where
\begin{equation*}
\begin{gathered}
g(s)= \int_{0}^{s}\lambda h(u(r)) \nabla r -A,\\
 A= \phi_{p}(u^{\Delta}(0))= -\frac{\lambda \beta}{1-\beta}\int_{0}^{\eta}
h(u(r))
\nabla r, \\
h(u(t))= \frac{\lambda f(u(t))}{( \int_{0}^{T} f(u(\tau ))\, \nabla
\tau )^{2}},\\
B=u(0)= \frac{1}{1-\beta} \left \{
\int_{0}^{T}\phi_{q}(g(s))\Delta s - \beta \int_{0}^{\eta}
\phi_{q}(g(s)) \Delta s \right  \}.
\end{gathered}
\end{equation*}
\end{lemma}

\begin{lemma}
Suppose $(H1)$ holds. Then a solution $u$ to \eqref{eq1}--\eqref{eq2}
satisfies $u(t) \geq 0$ for $t \in (0, T)_\mathbb{T}$.
\end{lemma}
\begin{proof} We have $A=\frac{-\lambda \beta}{1- \beta}
\int_{0}^{\eta}h(u(r)) \nabla r \leq 0 $. Then, $g(s)= \lambda
\int_{0}^{s} h(u(r))-A \geq 0$. It follows that $\phi_{p}(g(s)) \geq
0$. Since $0 < \beta < 1$, we also have
\begin{equation*}
\begin{split}
u(0)&= B \\
&= \frac{1}{1-\beta} \left \{ \int_{0}^{T}\phi_{q}(g(s))\Delta s
- \beta \int_{0}^{\eta} \phi_{q}(g(s)) \Delta s \right \}\\
& \geq \frac{1}{1-\beta} \left \{ \beta
\int_{0}^{T}\phi_{q}(g(s))\Delta s - \beta \int_{0}^{\eta}
\phi_{q}(g(s))\Delta s  \right \}\\
&\geq 0
\end{split}
\end{equation*}
and
\begin{equation*}
\begin{split}
u(T)&= u(0)-  \int_{0}^{T}\phi_{q}(g(s))\Delta s\\
&= \frac{-\beta}{1-\beta} \int_{0}^{\eta}\phi_{q}(g(s))\Delta s +
\frac{1}{1-\beta}\int_{0}^{T}\phi_{q}(g(s))\Delta s -
\int_{0}^{T}\phi_{q}(g(s))\Delta s\\
 & = \frac{-\beta}{1-\beta}
\int_{0}^{\eta}\phi_{q}(g(s))\Delta s + \frac{\beta}{1-\beta}
\int_{0}^{T} \phi_{q}(g(s))  \Delta s\\
& = \frac{\beta}{1-\beta}   \left \{
\int_{0}^{T}\phi_{q}(g(s))\Delta s-
\int_{0}^{\eta}\phi_{q}(g(s))\Delta s \right \}\\
&\geq 0.
\end{split}
\end{equation*}
If $t \in (0, T)_{\mathbb{T}}$, then
\begin{equation*}
\begin{split}
u(t)&= u(0)-  \int_{0}^{t}\phi_{q}(g(s))\Delta s\\
& \geq -  \int_{0}^{T}\phi_{q}(g(s))\Delta s + u(0)= u(T) \\
& \geq 0.
\end{split}
\end{equation*}
Consequently, $u(t) \geq 0$ for $t \in (0, T)_{\mathbb{T}}$.
\end{proof}

On the other hand, we have $\phi_{p}(u^{\Delta}(s))=
\phi_{p}(u^{\Delta}(0)) - \int_{0}^{s}\lambda h(u(r)) \nabla r
\leq 0$. Since $A= \phi_{p}(u^{\Delta}(0)) \leq 0$, then
$u^{\Delta} \leq 0$. This means that $||u|| = u(0)$,
$\inf_{t \in (0, T)_{\mathbb{T}}}u(t) = u(T)$. Moreover,
$\phi_{p}(u^{\Delta}(s))$ is non increasing, which implies with
the monotonicity of $\phi_{p}$ that $u^{\Delta}$ is a non
increasing function on $(0, T)_\mathbb{T}$. Hence,
$u$ is concave. In order to apply Theorem~\ref{thm1},
let us define the cone $P \subset E$ by
$$
P= \{ u \in E \ |\  u \mbox{ is nonnegative, decreasing on }
[0, T]_\mathbb{T} \mbox{ and } concave \mbox{ on } E \}\, .
$$
We also define the nonnegative continuous concave functional
$\alpha: P \rightarrow \mathbb{R}_{0}^{+}$ by
$$
\alpha(u) = \min_{t \in [\xi, T-\xi]_\mathbb{T}} u(t), \, \quad
\xi \in \left(0, \frac{T}{2}\right), \, \forall u \in P \, .
$$
It is easy to see that \eqref{eq1}--\eqref{eq2} has a solution
$u=u(t)$ if and only if $u$ is a fixed point of the operator $G: P
\rightarrow E$ defined by
\begin{equation}
\label{eq3}
Gu(t)=  -\int_{0}^{t}\phi_{q}\left (g(s) \right) \Delta s + B,
\end{equation}
where $g$ and $B$ are as in Lemma~\ref{lm1}.

\begin{lemma}\label{lm32}
Let $G$ be defined by \eqref{eq3}. Then,
\begin{itemize}
\item[$(i)$] $G(P) \subseteq P$;
\item[$(ii)$] $G: P \rightarrow P$ is completely continuous.
\end{itemize}
\end{lemma}
\begin{proof}
\begin{itemize}
\item[$(i)$] holds clearly from  above.
\item[$(ii)$] Suppose that $D \subseteq P$ is a bounded set and let $u \in D$. Then,
\begin{equation*}
\begin{split}
|Gu(t)| &= \left|-\int_{0}^{t}\phi_{q}\left (g(s) \right) \Delta s + B\right|\\
&\le \left|-\int_{0}^{t}\phi_{q} \left( \int_{0}^{s} \frac{\lambda
f(u(r))}{( \int_{0}^{T} f(u(\tau ))\, \nabla \tau )^{2}}\nabla r-A
\right)  \Delta s\right| + |B|\\
 & \leq \int_{0}^{T}\phi_{q}\left( \int_{0}^{s}\frac{\lambda \sup_{u
\in D} f(u)}{(T\inf_{u\in D}f(u))^{2}} \, \nabla r -A \right)\Delta
s + |B|\, ;
\end{split}
\end{equation*}
\end{itemize}
\begin{equation*}
\begin{split}
|A| &= \left|\frac{\lambda \beta}{1-\beta}\int_{0}^{\eta} h(u(r)) \nabla r\right| \\
& = \left|\frac{\lambda \beta}{1-\beta}\int_{0}^{\eta} \frac{
f(u(r))}{( \int_{0}^{T} f(u(\tau ))\, \nabla r )^{2}} \nabla r\right| \\
& \leq \frac{\lambda \beta}{1-\beta} \frac{ \sup_{u \in D}
f(u)}{(T\inf_{u\in D}f(u))^{2}} \, \, \eta.
\end{split}
\end{equation*}
In the same way, we have
\begin{equation*}
\begin{split}
|B| &\leq  \frac{1}{1-\beta} \int_{0}^{T} \phi_{q}(g(s))
\Delta s \\
& \leq  \frac{1}{1-\beta}\int_{0}^{T} \phi_{q}\left( \frac{\lambda
\sup_{u \in D}f(u)}{(T\inf_{u\in D}f(u))^{2}}\left(s+ \frac{\beta}{1-\beta}
\eta\right) \right) \Delta s .
\end{split}
\end{equation*}
It follows that
$$
|Gu(t)| \leq \int_{0}^{T} \phi_{q} \left ( \frac{\lambda \sup_{u \in D}
f(u)}{\left(T\inf_{u\in D}f(u)\right)^{2}}\left(s
+\frac{\beta \eta}{1-\beta}\right) \right) \Delta s +|B|.
$$
As a consequence, we get
\begin{equation*}
\begin{split}
\| Gu  \| & \leq \frac{2- \beta}{1-\beta}\int_{0}^{T} \phi_{q} \left
( \frac{\lambda \sup_{u \in D} f(u)}{(T\inf_{u\in
D}f(u))^{2}}\left(s+\frac{\beta \eta}{1-\beta}\right) \right)\\
 & \leq  \frac{2}{1-\beta}\phi_{q} \left( \frac{\lambda
\sup_{u \in D} f(u)}{(T\inf_{u\in D}f(u))^{2}} \right)\int_{0}^{T}
 \phi_{q} \left(s+\frac{\beta \eta}{1-\beta}\right)  \Delta s.
\end{split}
\end{equation*}
Then $G(D)$ is bounded on the whole bounded set $D$. Moreover,
if $t_{1}, t_{2} \in [0, T]_\mathbb{T}$ and $u \in D$, then
we have for a positive constant $c$ that
$$
|Gu(t_{2})-Gu(t_{1}))| \leq
\left|\int_{t_{1}}^{t_{2}}\phi_{q}(g(s)) \Delta s\right|
\leq c |t_{2}-t_{1}|.
$$
We see that the right hand side of the above inequality
goes uniformly to zero when $|t_{2}-t_{1}| \rightarrow 0$.
Then by a standard application of the Arzela--Ascoli theorem
we have that $G: P\rightarrow P$ is completely continuous.
\end{proof}
We can also easily obtain the following  properties:
\begin{lemma}
\begin{itemize}
\item[$(i)$] $\alpha(u) = u(T-\xi )  \leq \|u\|$  for all $u \in P$;
\item[$(ii)$] $\alpha(Gu)=(Gu)(T-\xi)$;
\item[$(iii)$]$ \|Gu\|= Gu(0)$.
\end{itemize}
\end{lemma}

We now state the main result of \S\ref{sec:mr:part1}.
\begin{theorem}
\label{thm:mr:secOrig1}
Suppose that $(H1)$ is verified and there exists
positive constants $a$, $b$, $c$, $d$ such that
$0 < \zeta a=a_{1} < b < d- T^{2} c \phi_{q}\left(\frac{1}{T}\right)
< d <\zeta c=c_{1}$, with
$\zeta = \frac{T^{2}}{1-\beta}\phi_{q}\left(\frac{1}{T}\right)$.
We further impose $f$ to satisfy the following hypotheses:
\begin{itemize}
\item[(H2)]$ \min_{0\leq u \leq a_{1}} f(u) \geq \frac{\lambda^{2}}{T(1-\beta)
\phi_{p}(a)}$ uniformly for all $t \in [0, T]_\mathbb{T}$;

\item[(H3)] $ \min_{0\leq u \leq c_{1}} f(u)
\geq \frac{\lambda^{2}}{T(1-\beta) \phi_{p}(c)}$
uniformly for all $t \in [0, T]_\mathbb{T}$;

\item[(H4)]$\min f(u)_{b \leq u \leq d } \geq \phi_{p}(bB_{1})$
uniformly for all $t \in [0, T]_\mathbb{T}$,  where
$$ B_{1}=\frac{(1- \beta)}{\beta \xi} |\phi_{p}(T-\xi)| \phi_{p}
\left(\frac{\lambda}{(T \sup_{b \leq u \leq d} f(u))^{2}} \right).
$$
\end{itemize}
Then the boundary value problem \eqref{eq1}--\eqref{eq2}
has at least three positive solutions
$u_{1}$, $u_{2}$, and $u_{3}$, verifying
$$ \|u_{1}\| < a\, , \quad
b < \alpha(u_{2})\, , \quad
\|u_{3}\| > a \, , \quad \mbox{ and } \,  \alpha(u_{3}) < b.$$
\end{theorem}
\begin{proof}
The proof passes by several lemmas.
We have already seen in Lemma~\ref{lm32}
that the operator $G$ is completely continuous.
We now show that
\begin{lemma}
$$
G\overline{P_{c_{1}}} \subset \overline{P_{c_{1}}}, \, \,
G\overline{P_{a_{1}}} \subset \overline{P_{a_{1}}}.
$$
\end{lemma}

\begin{proof}
Obviously, $G\overline{P_{a_{1}}} \subset P$.
Moreover, $\forall u \in \overline{P_{a_{1}}}$,
we have $0 \leq u(t) \leq a_{1}$.
On the other hand we have
$$
Gu(t)= -\int_{0}^{t}\phi_{q}\left (g(s) \right) \Delta s
+ \frac{1}{1-\beta} \int_{0}^{T}\phi_{q}\left (g(s) \right) \Delta s
-\frac{\beta}{1-\beta} \int_{0}^{\eta}\phi_{q}\left (g(s) \right) \Delta s \, ;
$$
and $\forall u \in G\overline{P_{a_{1}}}$
we have $0 \leq  u(t) \leq a_{1}$. Then,
$$
|Gu| \leq \frac{1}{1-\beta} \int_{0}^{T}\phi_{q}\left (g(s) \right) \Delta s \, .
$$
We have
$$
h(u(t))= \frac{\lambda f(u(t))}{( \int_{0}^{T} f(u(\tau ))\, \nabla \tau )^{2}}
$$
and
\begin{equation*}
\begin{split}
g(s)& = \int_{0}^{s}\lambda h(u(r)) \nabla r
+ \frac{\lambda \beta}{1-\beta}\int_{0}^{\eta}  h(u(r)) \nabla r\\
& \leq \left ( \lambda + \frac{\lambda \beta}{1-\beta}
\int_{0}^{T} h(u(r)) \nabla r  \right )\\
& \leq  \frac{\lambda}{1-\beta}\int_{0}^{T}  h(u(r)) \nabla r \, .
\end{split}
\end{equation*}
Using $(H2)$ it follows that
\begin{equation*}
\phi_{q}(g(s)) \leq  aT \phi_{q} \left(\frac{1}{T}\right).
\end{equation*}
Then we get
$$
|Gu| \leq a_{1} \mbox{ and } G\overline{P_{a_{1}}}
\subset \overline{P_{a_{1}}}\, .
$$
Similarly, using $(H3)$ we get $G\overline{P_{c_{1}}}
\subset \overline{P_{c_{1}}}$.
\end{proof}

\begin{lemma}
\label{lm37}
The set
$$
\{ u \in P(\alpha, b, d) \ | \  \alpha(u) > b \}
$$
is nonempty, and
$$
\alpha(Gu) > b,
\mbox{ if } u \in P(\alpha, b, d) \, .
$$
\end{lemma}

\begin{proof}
Let $u= \frac{b+d}{2}$. Then, $u \in P, \|u\|=\frac{b+d}{2}\leq d$
and $\alpha(u) \geq \frac{b+d}{2} > b$. The first part of the lemma is proved.
For $u \in P(\alpha, b, d)$ we have $b \leq u \leq d$.
If $t \in [\xi, T]_\mathbb{T}$, then
\begin{equation*}
\begin{split}
\alpha(Gu)&= (Gu)(T-\xi)\\
&=  -\int_{0}^{T-\xi}\phi_{q}\left (g(s) \right) \Delta s + B\\
&\geq \frac{\beta}{1-\beta} \int_{T-\xi}^{T}\phi_{q}\left (g(s) \right) \Delta s.
\end{split}
\end{equation*}
Since $A \leq 0$, we have by using $(H4)$
\begin{equation*}
\begin{split}
g(s) & = \lambda \int_{0}^{s} h(u(r)) \nabla r -A \\
&\geq \lambda \int_{0}^{s} h(u(r)) \nabla r \\
& \geq \lambda \int_{0}^{s} \frac{f(u)}{(T \sup_{b \leq u \leq
d}f(u))^{2}}\\
& \geq \lambda  \frac{(bB_{1})^{p-1}}{(T \sup_{b \leq u \leq
d}f(u))^{2}} s.
\end{split}
\end{equation*}
Using the fact that $\phi_{q}$ is nondecreasing we get
\begin{equation*}
\begin{split}
\phi_{q}(g(s))&  \geq  \phi_{q} \left ( \lambda
\frac{(bB_{1})^{p-1}}{(T \sup_{b \leq u \leq d}f(u))^{2}} s \right )\\
& \geq bB_{1} \phi_{q} \left(\frac{\lambda}{(T \sup_{b \leq u \leq d} f(u))^{2}}\right)
\phi_{q}(s).
\end{split}
\end{equation*}
Using the expression of $B_{1}$
\begin{equation*}
\begin{split}
\alpha(Gu)& \geq \frac{\beta}{1-\beta} bB_{1} \phi_{q}
\left(\frac{\lambda}{(T \sup_{b \leq u \leq d} f(u))^{2}}\right)
\int_{T-\xi}^{T} \phi_{q}(s)\Delta s\\
& \geq   bB_{1}\frac{\beta}{1-\beta} \phi_{q} \left(\frac{\lambda}{(T \sup_{b
\leq u \leq d} f(u))^{2}}\right) \phi_{q}(T-\xi) \xi\\
& \geq b.
\end{split}
\end{equation*}
\end{proof}

\begin{lemma}
\label{lm3.8}
For all $u \in P(\alpha, b, c_{1})$ with  $\|Gu\| > d$ one has
$$
\alpha(Gu) > b \, .
$$
\end{lemma}

\begin{proof}
If $u \in P(\alpha, b, c_{1})$ and $ \|Gu\| > d$, then $0 \leq u(t) \leq c_{1}$.
Using hypothesis (H3) and the fact that  $0 < \beta < 1$, it follows that
\begin{equation*}
\begin{split}
\alpha(Gu)&= Gu(T- \xi)\\
& =-\int_{0}^{T- \xi}\phi_{q}\left (g(s) \right) \Delta s + B \\
& \geq -\int_{0}^{T}\phi_{q}\left (g(s) \right) \Delta s + Gu(0)\\
& \geq \| Gu \| - T^{2} c \phi_{q}\left(\frac{1}{T}\right)\\
& \geq d- T^{2} c \phi_{q}\left(\frac{1}{T}\right)\\
& > b.
\end{split}
\end{equation*}
\end{proof}

Gathering Lemmas~\ref{lm1} to \ref{lm3.8}
and applying Theorem~\ref{thm1}, there exist at
least three positives solutions $u_{1}$, $u_{2}$,
and $u_{3}$ to \eqref{eq1}--\eqref{eq2} verifying
$$
\|u_{1}\| < a \, , \quad
b < \alpha(u_{2})\, , \quad
\|u_{3}\| > a \, , \quad
\mbox{ and } \, \alpha(u_{3}) < b.
$$
\end{proof}

\begin{example}
Let $\mathbb{T}=\left\{ 1-\left(\frac{1}{2}\right)^{\mathbb{N}_{0}}\right\} \cup \{ 1\}$,
where $\mathbb{N}_{0}$ denotes the set of all nonnegative
integers. Consider the $p$-Laplacian dynamic equation
\begin{equation}
\label{e6}
-\left ( \phi_{p}(u^{\Delta}(t))\right)^{\nabla}
= \frac{\lambda f(u(t))}{\left( \int_{0}^{T} f(u(\tau ))\, \nabla \tau \right)^{2}},
\quad  t \in  (0,T)_{\mathbb{T}} \, ,
\end{equation}
satisfying the boundary conditions
\begin{equation}
\label{e7}
\begin{gathered}
\phi_{p}\left(u^{\Delta}(0)\right) - \beta \left(
\phi_{p}\left(u^{\Delta}\left(\frac{1}{4}\right)\right)\right)=0,\\
u(1) -\beta u\left(\frac{1}{4}\right)=0,
\end{gathered}
\end{equation}
where $p=\frac{3}{2}$, $q=3$, $\eta = \frac{1}{4}$, $\beta=\frac{1}{2}$,
$\lambda=1$, $T=1$, and
$$
f(u)
=
\begin{cases}
2\sqrt{2}, & 0\leq u\leq 1, \\
4(u-1)+2\sqrt{2} & 1\leq u\leq \frac{3}{2}, \\
2+2\sqrt{2}  & \frac{3}{2}\leq u\leq 10,\\
2u + 2\sqrt{2}-18 & 10 \leq u\leq 16.
\end{cases}
$$
Choose $a_{1}=1=2a$, $b=\frac{3}{2}$, $c_{1}=16=2c$, and $d=10$.
It is easy to see that $\zeta=2$, $B_{1}=\frac{1}{2(2+\sqrt{2})}$, and
\begin{gather*}
\min \{ f(u):u\in [0,a_{1}]\} =2\sqrt{2}
\geq  \frac{\lambda^{2}}{T(1-\beta) \phi_{p}(a)}
= \frac{2}{\sqrt{a}} = 2\sqrt{2},\\
\min \{ f(u):u\in [0,c_{1}]\} =2\sqrt{2}
\geq \frac{\lambda^{2}}{T(1-\beta) \phi_{p}(c)}
= \frac{2}{\sqrt{c}}=\frac{1}{\sqrt{2}},\\
\min \{ f(u):u\in [b,d]\} =2+2\sqrt{2}
\geq \phi_{p}(bB_{1})=\sqrt{bB_{1}}
=\sqrt{\frac{3}{4(2+2\sqrt{2})}}.
\end{gather*}
Then, hypotheses $(H1)$--$(H4)$ are satisfied. Therefore,
by Theorem~\ref{thm:mr:secOrig1},
problem \eqref{e6}--\eqref{e7}
has at least three positive solutions.
\end{example}


\subsection{Quasilinear Elliptic Problem}
\label{sec:mr:part2}

We are interested in this section in the study of the following quasilinear
elliptic problem:
\begin{equation}
\label{eq:qep}
\begin{gathered}
-\left ( \phi_{p}(u^{\Delta}(t))\right)^{\nabla}= f(u(t)) +
h(t)\, , \quad  t \in  (0,T)_{\mathbb{T}} \, ,\\
u^{\Delta}(0)=0  \, , \quad u(T) - u(\eta)=0 \, ,
\end{gathered}
\end{equation}
where $\eta \in (0, T)_\mathbb{T}$.
We assume the following hypotheses:
\begin{itemize}
\item[(A1)] function $f: \mathbb{R} \rightarrow \mathbb{R}_0^{+}$ is
continuous;

\item[(A2)] function $h: (0,T)_\mathbb{T} \rightarrow \mathbb{R}_0^{+}$ is
left dense continuous, \textrm{i.e},
$$
h \in \mathbb{C}_{ld}\left((0,T)_\mathbb{T}, \mathbb{R}_0^{+}\right),
\quad \text{ and } h \in L^{\infty}.
$$
\end{itemize}
Similarly as in \S\ref{sec:mr:part1}, we prove existence of solutions by constructing
an operator whose fixed points are solutions to \eqref{eq:qep}.
The main ingredient is, again, the Leggett--Williams
fixed point theorem (Theorem~\ref{thm1}).
We can easily see that \eqref{eq:qep}
is equivalent to the integral equation
\begin{equation*}
u(t)= \phi_{q} \left (\int_{\eta}^{T}(f(u(r)+h(r)) \, \nabla r \right )
+\int_{0}^{t}\phi_{q}\left ( \int_{s}^{T} (f(u(r)+h(r)) \, \nabla r
\right ) \Delta s.
\end{equation*}
On the other hand, we have $-(\phi_{p}(u^{\Delta}))^{\nabla}
= f(u(t))+ h(t)$. Since $f$, $h \geq 0$, we have
$(\phi_{p}(u^{\Delta}))^{\nabla} \leq 0$ and
$(\phi_{p}(u^{\Delta}(t_2))) \leq
(\phi_{p}(u^{\Delta}(t_1)))$ for any $t_1,\, t_2 \in [0, T]_\mathbb{T}$
with $t_1 \leq t_2$. It follows that $u^{\Delta}(t_2) \leq
u^{\Delta}(t_1)$ for $t_1 \leq t_2$. Hence, $u^{\Delta} (t)$
is a decreasing function on $[0, T]_\mathbb{T}$. Then, $u$ is concave.
In order to apply Theorem~\ref{thm1} we define the cone
$$
P= \{ u \in E \ | \ u \mbox{ is nonnegative, increasing on }[0, T]_\mathbb{T} \, ,
\mbox{ and } concave \mbox{ on } E \}.
$$
For $\xi \in \left(0, \frac{T}{2}\right)$
we also define the nonnegative continuous concave functional
$\alpha : P \rightarrow \mathbb{R}_0^{+}$ by
$$
\alpha(u)= \min_{t \in [\xi, T-\xi]_\mathbb{T}} u(t)\, , \quad u \in P \, ,
$$
and the operator $F: P\rightarrow E$ by
$$
Fu(t) = \phi_{q} \left ( \int_{\eta}^{T} (f(u(r))+h(r)) \, \nabla r
\right )+ \int_{0}^{t} \phi_{q} \left ( \int_{s}^{T}(f(u(r)+h(r)) \,
\nabla r \right ) \Delta s.
$$
It is easy to see that \eqref{eq:qep} has a solution $u= u(t)$ if and
only if $u$ is a fixed point of the operator $F$.
For convenience, we  introduce the following notation:
\begin{equation*}
\begin{split}
\gamma &= (1+T)\phi_{q}(T)\, ,\\
A &= \frac{a- \alpha \|h\|_{\infty}^{1/p-1}}{\alpha a},\mbox{ where }
\alpha = \phi_{q}(2^{p-2}) \phi_{q}(T)(T+1) \, ,\\
B &= \phi_{p}(T- \eta)\, .
\end{split}
\end{equation*}

\begin{theorem}
\label{thm39}
Suppose that hypotheses $(A1)$ and $(A2)$ are satisfied;
there exist positive constants $a$, $b$, $c$, and $d$ with
$$
0 < \gamma a = a_{1} < b <d- 2^{p-2}(T-\xi)\phi_{q}(T- \xi)
\left( \|h\|^{\frac{1}{p-1}}+bB \right)< d < \gamma c = c_{1} \, ;
$$
and, in addition to (A1) and (A2), that $f$ satisfies
\begin{itemize}
\item[(A3)]   $\max_{0\leq u \leq a}f(u) \leq \phi_{p}(aA)$;

\item[(A4)] $\max_{0\leq u \leq c}f(u) \leq \phi_{p}(cA)$;

\item[(A5)] $\min_{b\leq u \leq d}f(u) \geq
\phi_{p}(bB)$.
\end{itemize}
Then problem \eqref{eq:qep} has at least three positive solutions
$u_{1}$, $u_{2}$, and $u_{3}$, verifying
$$ \|u_{1}\| < a\, , \quad
b < \alpha(u_{2})\, , \quad
\|u_{3}\| > a \, , \quad \mbox{ and } \,  \alpha(u_{3}) < b.$$
\end{theorem}

\begin{proof}
As done for Theorem~\ref{thm:mr:secOrig1},
the proof is divided in several steps. We first show that
$F: P \rightarrow P$ is completely continuous.
Indeed, $F$ is obviously continuous. Let
$$
U_{\delta} = \{ u \in P \ | \  \|u\|  \leq \delta \}.
$$
It is easy to see that for $u \in U_{\delta}$
there exists a constant $c> 0$ such that $|Fu(t)| \leq c$.
On the other hand, let $t_{1}, t_{2} \in (0, T)_\mathbb{T}$, $u \in U_{\delta}$.
Then there exists a positive constant $c$ such that
$$
|Fu(t_{2})- Fu(t_{1})|  \leq c |t_{2}- t_{1}| \, ,
$$
which converges uniformly to zero when $|t_{2}- t_{1}|$ tends to zero.
Using the Arzela--Ascoli theorem we conclude that
$F: P\rightarrow P$ is completely continuous.

We now show that
$$
F\overline{P_{c_{1}}} \subset \overline{P_{c_{1}}}\, ,
\quad F\overline{P_{a_{1}}} \subset \overline{P_{a_{1}}}\, .
$$
For all $u \in \overline{P_{a_{1}}}$ we have $0 \leq u \leq a_{1}$ and
\begin{equation*}
\begin{split}
\|F(u)\| &  \leq \phi_{q} \left ( \int_{\eta}^{T}((aA)^{p-1}+
\|h\|_{\infty}) \nabla r \right ) \\
& \qquad+ \int_{0}^{T} \phi_{q} \left
(\int_{s}^{T}((aA)^{p-1}+ \|h\|_{\infty}) \nabla r \right )
\Delta s\\
&\leq \phi_{q}\left ( ((aA)^{p-1}+ \|h\|_{\infty})(T- \eta) \right) \\
& \qquad + \int_{0}^{T} \phi_{q} \left ( ((aA)^{p-1}+ \|h\|_{\infty})(T- s)
\right ) \Delta s\\
&\leq \phi_{q}\left ( (aA)^{p-1}+ \|h\|_{\infty}\right )
\phi_{q}(T) \\
& \qquad + \phi_{q}\left ( (aA)^{p-1}+ \|h\|_{\infty}\right )
\int_{0}^{T} \phi_{q}(T-s) \Delta s\\
&\leq \phi_{q}\left ( (aA)^{p-1}+
(\|h\|_{\infty}^{1/p-1})^{p-1}\right ) \phi_{q}(T) \\
& \qquad + \phi_{q}\left (
(aA)^{p-1}+ (\|h\|_{\infty}^{1/p-1})^{p-1}\right )
\phi_{q}(T) T \\
& \leq \phi_{q}\left ( (aA)^{p-1}+
(\|h\|_{\infty}^{1/p-1})^{p-1}\right ) \phi_{q}(T) (T+1).
\end{split}
\end{equation*}
Using the elementary inequality
$$
x^{p} + y^{p} \leq 2^{p-1}(x+y)^{p},
$$
and the form of $A$, it follows that
\begin{equation*}
\begin{split}
\|F(u)\| &  \leq \phi_{q}(T+1)(2^{p-2})(aA+
\|h\|_{\infty}^{1/p-1})\\
& \leq  \phi_{q}(T+1)(2^{p-2})a =\gamma a = a_{1}.
\end{split}
\end{equation*}
Then $F\overline{P_{a_{1}}} \subset \overline{P_{a_{1}}}$.
In a similar way we prove that
$F\overline{P_{c_{1}}} \subset \overline{P_{c_{1}}}$.

Our following step consists in show that
$$
\{ u \in P(\alpha, b, d) \ | \ \alpha(u) > b \} \neq \emptyset
$$
and
\begin{equation}
\label{eq:ntp}
\alpha(Fu) > b, \mbox{ if } u \in P(\alpha, b, d).
\end{equation}
The first point is obvious. Let us prove \eqref{eq:ntp}.
For $u \in P(\alpha, b, d)$ we have $b \leq u \leq d$,
if $t \in [\xi, T]_\mathbb{T}$. Then, using $(A2)$ we have
\begin{equation*}
\begin{split}
\alpha(Fu) &= Fu(\xi)\\
& \geq \phi_{q}\left ( \int_{\eta}^{T} f(u(r) \, \nabla r
\right ) + \int_{0}^{\xi} \phi_{q} \left ( \int_{s}^{T}(f(u(r)) \,
\nabla r
\right ) \Delta s \\
& \geq \phi_{q}\left ( \int_{\eta}^{T} f(u(r) \, \nabla r
\right )\\
& \geq bB\phi_{q}(T-\xi)\\
& \geq b.
\end{split}
\end{equation*}
Finally we prove that
$\alpha(Fu) > b$ for all  $u \in P(\alpha, b, c_{1})$ and $\|Fu\| > d$:
\begin{equation*}
\begin{split}
\alpha(Fu) &= Fu(\xi)\\
&= \phi_{q} \left ( \int_{\eta}^{T} (f(u(r))+h(r)) \, \nabla r
\right )+ \int_{0}^{\xi} \phi_{q} \left ( \int_{s}^{T}(f(u(r)+h(r)) \,
\nabla r \right ) \Delta s\\
&= \phi_{q} \left ( \int_{\eta}^{T} (f(u(r))+h(r)) \, \nabla r
\right )+ \int_{0}^{T} \phi_{q} \left ( \int_{s}^{T}(f(u(r)+h(r)) \,
\nabla r \right ) \Delta s\\
& \qquad - \int_{\xi}^{T} \phi_{q} \left ( \int_{s}^{T}(f(u(r)+h(r)) \, \nabla r \right ) \Delta s\\
& \geq \|Fu\| - \int_{\xi}^{T} \phi_{q} \left ( \int_{s}^{T}(f(u(r)+h(r)) \, \nabla r \right ) \Delta s\\
& \geq \|Fu\| - \int_{\xi}^{T} \phi_{q} \left ( \int_{\xi}^{T}(f(u(r)+h(r)) \, \nabla r \right ) \Delta s\\
& \geq \|Fu\| - \int_{\xi}^{T} \phi_{q}\left( (T- \xi)(\|h\|+\phi_{p}(bB)  \right)\\
& \geq \|Fu\| -(T-\xi)\phi_{q}(T- \xi) \phi_{q}\left( \|h\|+\phi_{p}(bB) \right).
\end{split}
\end{equation*}
Using again the elementary inequality
$x^{p} + y^{p} \leq 2^{p-1}(x+y)^{p}$
we get that
\begin{equation*}
\begin{split}
\alpha(Fu) & \geq \|Fu\| - (T-\xi)\phi_{q}(T- \xi) \left( \|h\|^{\frac{1}{p-1}}+bB \right)\\
&  \geq d- 2^{p-2}(T-\xi)\phi_{q}(T- \xi) \left( \|h\|^{\frac{1}{p-1}}+ bB  \right )\\
& \geq b.
\end{split}
\end{equation*}
By Theorem~\ref{thm1} there exist at least three positive solutions
$u_{1}$, $u_{2}$, and $u_{3}$ to \eqref{eq:qep} satisfying
$\|u_{1}\| < a$, $b < \alpha(u_{2})$, $\|u_{3}\| > a$,
and $\alpha(u_{3}) < b$.
\end{proof}

\begin{example}
Let $\mathbb{T}=\left\{ 1-\left(\frac{1}{2}\right)^{\mathbb{N}_{0}}\right\}
\cup \{ 1\}$, where $\mathbb{N}_{0}$ denotes the set of all nonnegative
integers. Consider the $p$-Laplacian dynamic equation
\begin{equation}
\label{e6.1}
[\phi _p(u^{\Delta }(t))] ^{\nabla }+f(u(t))=0,
\quad t\in [0,1]_{\mathbb{T}},
\end{equation}
satisfying the boundary conditions
\begin{equation}
\label{e6.2}
u(1)-u\left(\frac{1}{2}\right)=0,
\quad u^{\triangle }(0)=0,
\end{equation}
where $p=\frac{3}{2}$, $q=3$,
$a(t)\equiv 1$, $h \equiv 0$, $T=1$, and
$$
f(u)=\begin{cases}
\frac{\sqrt{2}}{2}, & 0\leq u\leq \frac{1}{2}, \\
4(u-\frac{1}{2})+\frac{\sqrt{2}}{2} & \frac{1}{2}\leq u\leq \frac{3}{2},\\
4+\frac{\sqrt{2}}{2} & \frac{3}{2} \leq u \leq 25.
\end{cases}
$$
Choose $a=\frac{1}{2}$, $b=\frac{3}{2}$, $c=25$, and $d=3$. It is easy
to see that $\gamma=2$, $A=1$, $B=\frac{\sqrt{2}}{2}$, $\alpha=1$, and
\begin{gather*}
\max \left\{ f(u):u\in [0,\frac{1}{2}]\right\} =\frac{\sqrt{2}}{2}
\leq (aA)^{p-1} =\sqrt{a}=\frac{\sqrt{2}}{2},\\
\max \{ f(u):u\in [0,25]\} =4+\frac{\sqrt{2}}{2}\simeq 4,707
\leq (cA)^{p-1}=c^{p-1}=\sqrt{c}= 5 ,\\
\min \left\{ f(u):u\in \left[ \frac{3}{2},3\right]\right\}
=4+\frac{\sqrt{2}}{2} \simeq 4,707 \geq (bB)^{p-1}=\sqrt{bB}\simeq 1,02.
\end{gather*}
Therefore, by Theorem~\ref{thm39}, problem \eqref{e6.1}--\eqref{e6.2}
has at least three positive solutions.
\end{example}


\subsection{A $p$-Laplacian Functional Dynamic Equation on Time Scales
with Delay}
\label{sec:mr:part3}

Let $\mathbb{T}$ be a time scale with $0, T \in \mathbb{T}_\kappa^\kappa$,
$-r \in \mathbb{T}$ with $-r \le 0 < T$.
We are concerned in this section with the existence of positive solutions
to the $p$-Laplacian dynamic equation
\begin{equation}
\label{equa1:subsec3}
\begin{gathered}
\left(\phi_{p}(u^{\Delta}(t))\right)^\nabla
+ \lambda a(t) f(u(t), u(\omega(t)))=0 \, ,
\quad t \in (0,T)_\mathbb{T} \, ,\\
u(t)= \psi(t)\, , \quad t \in [-r, 0]_\mathbb{T}\, ,
\quad u(0)-B_{0}(u^{\Delta}(0))=0\, ,
\quad u^{\Delta}(T)=0 \, ,
\end{gathered}
\end{equation}
where $\lambda > 0$.
We define $ X= \mathbb{C}_{ld}([0, T]_\mathbb{T}, \mathbb{R})$,
which is a Banach space with the maximum norm
$\|u\|= \max_{[0, T]_\mathbb{T}}|u(t)|$.
We note that $u$ is a solution to \eqref{equa1:subsec3}
if and only if
\begin{equation*}
u(t)
=\begin{cases}
B_0 \left(\phi_q\left(\int_{0}^{T}\lambda a(r)f(u(r),u(\omega(r)))\nabla r
\right)\right)\\
\quad  + \int_{0}^{t}\phi_q\left(\int_{s}^{T}\lambda a(r)f(u(r),u(\omega(r)))\nabla r
\right)\Delta s  & \text{if } t\in [0,T]_\mathbb{T},\\
\psi(t) & \text{if } t\in[-r,0]_\mathbb{T} \, .
\end{cases}
\end{equation*}
Let
$$
K = \{ u \in X \ | \ u \mbox{ is nonnegative and concave on } E \}.
$$
Clearly $K$ is a cone in the Banach space $X$. For each $u \in X$, we extend
$u$ to $[-r, 0]_\mathbb{T}$ with $u(t)=\psi(t)$ for $t \in [-r, 0]_\mathbb{T}$.
We also define the nonnegative continuous concave functional
$\alpha: P \rightarrow \mathbb{R}_0^+$ by
$$
\alpha(u) = \min_{t \in [\xi, T-\xi]_\mathbb{T}} u(t) \, , \quad
\xi \in \left(0, \frac{T}{2}\right)\, , \quad \forall u \in K.
$$
For $t \in  [0, T]_\mathbb{T}$, define $Q: K \rightarrow X$ as
\begin{equation}
\label{eq:def:Qu}
\begin{aligned}
Q u(t)&= B_0\left(\phi_q\left(\int_{0}^{T}\lambda a(r)f\left(u(r),u(\omega(r))\right)
\nabla r\right)\right)\\
&\qquad  + \int_{0}^{t}\phi_q\left(\int_{s}^{T}
\lambda a(r)f(u(r),u(\omega(r)))\nabla r\right)\Delta s \, .
\end{aligned}
\end{equation}

\begin{lemma}
Let $u_{1}$ be a fixed point of $Q$ in the cone $K$. Define
\begin{equation}
\label{u:ps}
u(t)=
\begin{cases}
 u_{1},  & t \in [0, T]_\mathbb{T}, \\
\psi (t),  & t \in [-r, 0]_\mathbb{T}.
\end{cases}
\end{equation}
It follows that \eqref{u:ps}
is a positive solution to \eqref{equa1:subsec3} satisfying
\begin{equation}
\label{equa2}
\| Q u\| \leq (T+\gamma)\lambda^{q-1}\phi_q\left(\int_{0}^{T}
a(r)f\left(u(r),u(\omega(r))\right)\nabla r\right)
\mbox{ for } t \in [0, T]_\mathbb{T}.
\end{equation}
\end{lemma}

\begin{proof}
\begin{equation*}
\begin{aligned}
\|Q u\|&=(Q u)(T)\\
&= B_0 \left(\phi_q\left(\int_{0}^{T}\lambda a(r)f(u(r),u(\omega(r)))\nabla
r\right)\right)\\
&\quad + \int_{0}^{T}\phi_q\left(\int_{s}^{T} \lambda
a(r)f\left(u(r),u(\omega(r))\right)
\nabla r\right)\Delta s\\
&\leq  (T+\gamma)\lambda^{q-1}\phi_q\left(\int_{0}^{T}
a(r)f\left(u(r),u(\omega(r))\right)\nabla r\right) \, .
\end{aligned}
\end{equation*}
\end{proof}

From \eqref{eq:def:Qu} and \eqref{equa2} it follows that
\begin{itemize}
\item[(i)]  $Q (K) \subset K$;
\item[(ii)] $Q : K\to K$ is completely continuous;
\item[(iii)] $ u(t)\geq \frac{\delta}{T+\gamma}\|u\|$,
$t\in[0,T]_\mathbb{T}$.
\end{itemize}

Depending on the signature of the delay $\omega$,
we set the following two subsets of $[0, T]_\mathbb{T}$:
$$
Y_1:=\{t\in[0,T]_\mathbb{T} \ | \ \omega (t)<0\};
\quad  Y_2:=\{t\in[0,T]_\mathbb{T} \ | \ \omega(t) \geq 0\}.
$$
In the remainder of this section, we suppose that $Y_1$ is nonempty and
$\int_{Y_1}a(r)\nabla r>0$. For convenience we also denote
\begin{gather*}
l := \frac{\phi_p \big(\int_{0}^{T}a(r)\nabla r\big)}{\lambda^{q-1}(T+\gamma)},
\quad
m := \frac{\phi_p \big(\int_{0}^{T}a(r)\nabla r\big)}{\delta \lambda^{q-1}}.
\end{gather*}

\begin{theorem}
\label{thm:5.2}
Suppose that there exist positive constants $a$, $b$, $c$, $d$ such that
$0 < a < b  < \frac{\delta d}{T+ \gamma}< d < c$.
Assume that the following hypotheses $(C1)$--$(C8)$ hold:
\begin{itemize}
\item[(C1)] $f:\mathbb{R}_0^{+} \times \mathbb{R}_0^{+}
\rightarrow \mathbb{R}_0^{+}$ is continuous;

\item[(C2)] function $a: (0,T)_\mathbb{T}
\rightarrow \mathbb{R}_0^{+}$ is left dense continuous;

\item[(C3)] $\psi : [-r, 0]_\mathbb{T}
\rightarrow \mathbb{R}_0^{+}$ is continuous;

\item[(C4)] $\omega: [0, T]_\mathbb{T}
\rightarrow [-r, T]_\mathbb{T}$ is continuous,
$\omega(t)\leq t$ for all $t$;

\item[(C5)] $B_{0}: \mathbb{R}
\rightarrow \mathbb{R}$ is continuous and
there are $0< \delta \leq \gamma$ such that
$$
\delta s \leq B_{0}(s) \leq \gamma s \mbox{ for } s \in \mathbb{R}_0^{+}.
$$

\item[(C6)] $\lim_{x\to 0^+}\frac{f(x,\psi(s))}{x^{p-1}}<l^{p-1}$,
uniformly in $s\in [-r,0]_\mathbb{T}$;

\item[(C7)]
$\lim_{x_1\to 0^+;x_2\to 0^+}\frac{f(x_1,x_2)}{\max\{x_1^{p-1},x_2^{p-1}\}}
<l^{p-1}$;

\item[(C8)] $\lim_{x\to\infty}\frac{f(x,\psi(s))}{x^{p-1}}>m^{p-1}$,
uniformly in $s\in [-r,0]_\mathbb{T}$.
\end{itemize}
Then, for each $0 < \lambda < \infty$ the boundary value problem \eqref{equa1:subsec3}
has at least three positive solutions $u_{1}$, $u_{2}$, and $u_{3}$ verifying
$$
\|u_{1}\| < a\, , \quad
b < \alpha(u_{2})\, , \quad
\|u_{3}\| > a \, , \quad  \alpha(u_{3}) < b\, .
$$
\end{theorem}

\begin{proof}
The proof passes by three lemmas.
\begin{lemma}
The following relations hold:
$$
Q\overline{P_{a}} \subset \overline{P_{a}} \, ,
\quad Q\overline{P_{c}} \subset \overline{P_{c}} \, .
$$
\end{lemma}
\begin{proof}
Using condition $(C6)$ for $\varepsilon_1>0$
such that $0<x\leq \varepsilon_1$, we have
$$
f(x,\psi(s))< (lx)^{p-1} \quad \mbox {for each }  s\in [-r,0]_\mathbb{T}.
$$
Applying condition $(C7)$ we get
$$
f(x_1,x_2)<\max\{x_1^{p-1},x_2^{p-1}\}l^{p-1}
$$
for $\varepsilon_2>0$ such that $0<x_1\leq \varepsilon_2$, $0<x_2\leq \varepsilon_2$.
Put $\varepsilon = \min\{\varepsilon_1,\varepsilon_2\}$. Then for $\|u\| \leq a$ and
from \eqref{equa2} we have
\begin{equation*}
\begin{aligned}
& |Q  u| \leq \|Q u\|\\
&\leq (T+\gamma)\lambda^{q-1}\phi_q\left(\int_{0}^{T}
a(r)f(u(r),u(\omega(r)))\nabla r\right)\\
&=(T+\gamma)\lambda^{q-1}\left[\phi_q\left(\int_{Y_1}
a(r)f\left(u(r),\psi(\omega(r))\right)
\nabla r+\int_{Y_2}
a(r)f\left(u(r),u(\omega(r))\right)
\nabla r\right)\right]\\
&\leq l(T+\gamma)\lambda^{q-1}\|u\| \phi_q
\left(\int_{0}^{T}a(r)\nabla r\right)\\
& \leq  l(T+\gamma)\lambda^{q-1}a
\phi_q\left(\int_{0}^{T}a(r)\nabla r\right)\\
&=a \, .
\end{aligned}
\end{equation*}
Then $ Q\overline{P_{a}} \subset \overline{P_{a}}$.
Similarly one can show that
$Q\overline{P_{c}} \subset \overline{P_{c}}$.
\end{proof}

\begin{lemma}
The set
$$\{ u \in P(\alpha, b, d) \ | \  \alpha(u) > b \} $$
is nonempty, and
$$ \alpha(Qu) > b, \mbox{ if } u \in P(\alpha, b, d).$$
\end{lemma}
\begin{proof}
Applying hypothesis $(C8)$ we have
$$
f(u, \psi(s)) > \phi_{p}(mu) \mbox{ for each } s \in [-r, 0]_\mathbb{T}.
$$
We also have $u(t) \geq \frac{\delta}{T+ \gamma} \| u\|$.
Let $u \in P(\alpha, b, d)$. Then, $b \leq u(t) \leq d$. Hence,
\begin{equation*}
\begin{aligned}
\alpha(Q u)&= (Q u)(T-\xi)\\
& \geq \frac{\delta}{T+ \gamma} \| Q u\|\\
&\geq \delta\phi_q\left(\int_{0}^{T}\lambda a(r)f(u(r),u(\omega(r)))\nabla r\right)\\
& \geq \delta\lambda^{q-1}\phi_q\left(\int_{Y_1} a(r)f(u(r),\psi(\omega(r)))
\nabla r + \int_{Y_2} a(r)f(u(r),\psi(\omega(r)))
\nabla r\right) \\
&\geq \delta\lambda^{q-1}\phi_q\left(\int_{Y_1} a(r)f\left(u(r),\psi\left(\omega(r)\right)\right)
\nabla r\right)\\
&\geq m\delta\lambda^{q-1}\min_{t\in Y_1}\{u(t)\}\phi_q
 \left(\int_{Y_1} a(r)\nabla r\right)\\
& \geq bm \delta  \lambda^{q-1}\phi_q\left(\int_{Y_1} a(r) \nabla r\right)\\
& \geq b \, .
\end{aligned}
\end{equation*}
\end{proof}
\begin{lemma}
For all $u \in P(\alpha, b, c)$ and  $\|Qu\| > d$ one has
$\alpha (Qu) > b$.
\end{lemma}

\begin{proof}
Using the fact that $\delta \leq T + \gamma$, we have
\begin{equation*}
\begin{aligned}
\alpha(Q u)&= (Q u)(T-\xi)\\
&\geq \frac{\delta}{T + \gamma}  \| Q u\|\\
&\geq \frac{\delta d}{T + \gamma} \\
 &\geq b.
 \end{aligned}
\end{equation*}
\end{proof}
Applying the Leggett--Williams theorem (Theorem~\ref{thm1}),
the proof of Theorem~\ref{thm:5.2} is complete.
\end{proof}

\begin{example}
Let $\mathbb{T}=\left[-\frac{3}{4},-\frac{1}{4}\right]
\cup \left\{ 0,\frac{3}{4}\right\}
\cup \left\{(\frac{1}{2})^{\mathbb{N}_0}\right\}$,
where $\mathbb{N}_0$ denotes the
set of all nonnegative integers. Consider the following
$p$-Laplacian functional dynamic equation on the time scale
$\mathbb{T}$:
\begin{equation}
\label{example3}
\begin{gathered}
\hspace{0.1cm}[\Phi _p(u^{\Delta }(t))] ^{\nabla
}+(u_{1}+u_{2})^{2}=0,\quad t\in (0,1)_{\mathbb{T}}, \\
\psi (t)\equiv 0,
\quad t\in \left[-\frac{3}{4},0\right]_{\mathbb{T}},\\
u(0)-B_0\left(u^{\Delta }\left(\frac{1}{4}\right)\right)=0,
\quad u^{\Delta}(1)=0,
\end{gathered}
\end{equation}
where $T=1$, $p=\frac{3}{2}$, $q=3$, $a(t)\equiv 1,B_{0}(s)=s$,
$w(t) :[0,1]_{\mathbb{T}}\to [-\frac{3}{4},1]_{\mathbb{T}}$ with
$w(t)=t-\frac{3}{4}$, $r=\frac{3}{4}$, $\eta=\frac{1}{4}$,
$l=\frac{1}{2}$, $m=1$, and $f(u,\psi (t))=u^{2}$,
$f(u_1,u_2)=\left(u_{1}+u_{2}\right)^{2}$. We deduce that
$Y_1=[0,\frac{3}{4})_{\mathbb{T}}$,
$Y_2=\left[\frac{3}{4},1\right]_{\mathbb{T}}$. It is easy to
see that hypotheses $(C1)$--$(C5)$ are verified. On the other hand,
notice that $\lim_{x\to 0^+}\frac{f(x,\psi(s))}{x^{p-1}}= 0
<l^{p-1}$ and $\lim_{x\to\infty}\frac{f(x,\psi(s))}{x^{p-1}}
=+\infty >m^{p-1}$. Thus, hypotheses $(C6)$--$(C8)$ are obviously
satisfied. Then, by Theorem~\ref{thm:5.2}, the problem
\eqref{example3} has at least three positive solutions of the form
\[
u(t)
=\begin{cases}
u_i(t), & t\in [0,1]_{\mathbb{T}},\quad i=1,2,3, \\
\psi (t), & t\in [-\frac{3}{4},0]_{\mathbb{T}}.
\end{cases}
\]
\end{example}


\subsection*{Acknowledgments}

This work was supported by FEDER funds through
COMPETE --- Operational Programme Factors of Competitiveness
(``Programa Operacional Factores de Competitividade'')
and by Portuguese funds through the
{\it Center for Research and Development
in Mathematics and Applications} (University of Aveiro)
and the Portuguese Foundation for Science and Technology
(``FCT --- Funda\c{c}\~{a}o para a Ci\^{e}ncia e a Tecnologia''),
within project PEst-C/MAT/UI4106/2011
with COMPETE number FCOMP-01-0124-FEDER-022690.
The authors were also supported by the project
\emph{New Explorations in Control Theory Through Advanced Research} (NECTAR)
cofinanced by FCT, Portugal, and the \emph{Centre National de la Recherche
Scientifique et Technique} (CNRST), Morocco.



\end{document}